\documentclass[12pt,psamsfonts]{amsart}
\usepackage{amsmath,amsthm,amsfonts,amssymb}
\usepackage{graphicx}
\usepackage[all,knot]{xy}
\xyoption{arc}

\usepackage{eucal}

\newcommand{\A}{\mathcal{A}}
\newcommand{\D}{\mathcal{D}}

\newcommand{\ep}{\epsilon}
\newcommand{\om}{\omega}
\newcommand{\lan}{\langle}
\newcommand{\ra}{\rangle}

\newcommand{\F}{\mathcal{F}}
\newcommand{\ssl}{\mathfrak{sl}}

\newcommand{\mF}{\mathbb{F}}

\newcommand{\la}{{\lambda}}

\DeclareMathOperator{\Rep}{Rep}
\DeclareMathOperator{\Ann}{Ann}

\DeclareMathOperator{\End}{End}

\DeclareMathOperator{\im}{i}

\DeclareMathOperator{\tr}{tr}
\DeclareMathOperator{\Tr}{Tr}

\newcommand{\one}{\mathbf{1}}
\newcommand{\C}{\mathbb C}
\newcommand{\CC}{\mathcal{C}}
\newcommand{\Z}{\mathbb Z}
\newcommand{\ot}{\otimes}
\newcommand{\B}{\mathcal{B}}
\newcommand{\mH}{\mathcal{H}}

\numberwithin{equation}{section}

\newtheorem{theorem}{Theorem}[section]

\newtheorem{lemma}{Lemma}[section]

\newtheorem{question}[equation]{Question}

\theoremstyle{definition}

\newtheorem{remark}{Remark}[section]

\newtheorem{definition}[equation]{Definition}

\begin{document}

\title[Quaternionic braid representation]
{A quaternionic braid representation (after Goldschmidt and Jones)}

\author{Eric C.\ Rowell}
\address{Department of Mathematics,
    Texas A\&M University, College Station, TX \textup{77843}, U.S.A.}
   \email{rowell@math.tamu.edu}
\thanks{The author was partially supported by NSA grant H98230-10-1-0215 and benefitted from discussions with V. F. R. Jones, M. J. Larsen, V. Ostrik, D. Nikshych, F. Sottile and H. Wenzl.  Handwritten notes by D. Goldschmidt and V. Jones were very useful.}
\begin{abstract}
We show that the braid group representations associated with the $(3,6)$-quotients of the Hecke algebras factor over a finite group.  This was known to experts going back to the 1980s, but a proof has never appeared in print.  Our proof uses an unpublished quaternionic representation of the braid group due to Goldschmidt and Jones.  Possible topological and categorical generalizations are discussed.
\end{abstract}
\maketitle

\section{Introduction}

Jones analyzed the images of the braid group representations obtained from Temperley-Lieb algebras in \cite{jones86} where, in particular, he determined when the braid group images are finite or not.  Braid group representations with finite image were also recognized in \cite{jones89} and \cite{GJ}.  Some 15 years later the problem of determining the closure of the image of braid group representations associated with Hecke algebras played a critical role in analyzing the computational power of the topological model for quantum computation \cite{FLW}.  Following these developments the author and collaborators analyzed braid group representations associated with $BMW$-algebras \cite{LRW} and twisted doubles of finite groups \cite{ERW}.

Partially motivated by empirical evidence the author conjectured that the braid group representations associated with an object $X$ in a braided fusion category $\CC$ has finite image if, and only if, the Frobenius-Perron dimension of $\CC$ is integral (see eg. \cite[Conjecture 6.6]{RSW}).  In \cite{NR,RUMA} various instances of this conjecture were verified.  This current work verifies this conjecture for the braided fusion category $\CC(\ssl_3,6)$ obtained from the representation category of the quantum group $U_q\ssl_3$ at $q=e^{\pi \im/6}$ (see \cite{Rsurvey} for details and notation).

More generally, Jimbo's \cite{Jm} quantum Schur-Weyl duality establishes a relationship between the modular categories $\CC(\ssl_k,\ell)$ obtained from the quantum group $U_q\ssl_k$ at $q=e^{\pi \im/\ell}$ and certain semisimple quotients $\mH_n(k,\ell)$ of specialized Hecke algebras $\mH_n(q)$ (defined below).
That is, if we denote by $X\in\CC(\ssl_k,\ell)$ the simple object analogous to the vector representation of $\ssl_k$ then there is an isomorphism $\mH_n(k,\ell)\cong\End(X^{\ot n})$ induced by $g_i\rightarrow I_X^{\ot i-1}\ot c_{X,X}\ot I^{\ot n-i-1}$.

In particular, the braid group representations associated with the modular category $\CC(\ssl_3,6)$ are the same as those obtained from $\mH_n(3,6)$.
It is known that braid group representations obtained from $\mH_n(3,6)$ have finite image (mentioned in \cite{FLW,LR,NR}) but a proof has never appeared in print.  This fact was discovered by Goldschmidt and Jones during the writing of \cite{GJ} and independently by Larsen during the writing of \cite{FLW}.  We benefitted from the notes of Goldschmidt and Jones containing the description of the quaternionic braid representation below.  Our techniques follow closely those of \cite{jones86,jones89,LR2}.

The rest of the paper is organized into three sections.  In Section \ref{Hecke} we recall some notation and facts about Hecke algebras and their quotients.  The main results are in Section \ref{quat}, and in Section \ref{disc} we indicate how the category $\CC(\ssl_3,6)$ is exceptional from topological and categorical points of view.

\section{Hecke Algebras}\label{Hecke}
We extract the necessary definitions and results from \cite{W} that we will need in the sequel.
\begin{definition}
The \emph{Hecke algebra} $\mH_n(q)$ for $q\in\C$ is the $\C$-algebra with generators $g_1,\ldots, g_{n-1}$ satisfying relations:

\begin{enumerate}
\item[$(H1)^\prime$] $g_ig_{i+1}g_i=g_{i+1}g_ig_{i+1}$ for $1\leq i\leq n-2$
\item[$(H2)^\prime$] $g_ig_j=g_jg_i$ for $|i-j|>1$
 \item[$(H3)^\prime$] $(g_i+1)(g_i-q)=0$
\end{enumerate}
\end{definition}

Technically, $\mH_n(q)$ is the Hecke algebra of type $A$, but we will not be considering other types so we suppress this distinction.  One immediately observes that $\mH_n(q)$ is the quotient of the braid group algebra $\C\B_n$ by the relation $(H3)^\prime$.
$\mH_n(q)$ may also be described in terms of the generators $e_i:=\frac{(q-g_i)}{(1+q)}$, which satisfy:
\begin{enumerate}
 \item[$(H1)$] $e_i^2=e_i$
 \item[$(H2)$] $e_ie_j=e_je_i$ for $|i-j|>1$
 \item[$(H3)$] $e_ie_{i+1}e_i-q/(1+q)^2e_i=e_{i+1}e_{i}e_{i+1}-q/(1+q)^2e_{i+1}$ for $1\leq i\leq n-2$
\end{enumerate}

For any $\eta\in\C$, Ocneanu \cite{FYHLMO} showed that one may uniquely define a linear functional $\tr$ on $\mH_\infty(q):=\cup_{n=1}^\infty \mH_n(q)$
satisfying
\begin{enumerate}
 \item $\tr(1)=1$
\item $\tr(ab)=\tr(ba)$
\item $\tr(xe_n)=\eta\tr(x)$ for any $x\in\mH_n(q)$
\end{enumerate}
Any linear function on $\mH_\infty$ satisfying these conditions is called a \emph{Markov trace} and is determined by the value $\eta=\tr(e_1)$.

Now suppose that $q=e^{2\pi i/\ell}$ and $\eta=\frac{(1-q^{1-k})}{(1+q)(1-q^k)}$ for some integers $k<\ell$.  Then for each $n$, the (semisimple) quotient of $\mH_n(q)$ by the annihilator of the restriction of the trace $\mH_n(q)/\Ann(\tr)$ is called the \textit{$(k,\ell)$-quotient}.  We will denote this quotient by $\mH_n(k,\ell)$ for convenience.  Wenzl \cite{W} has shown that $\mH_n(k,\ell)$ is semisimple and described the irreducible representations $\rho_\la^{(k,\ell)}$ where $\la$ is a \emph{$(k,\ell)$-admissible} Young diagrams of size $n$.  Here a Young diagram $\la$ is $(k,\ell)$-admissible if $\la$ has at most $k$ rows and $\la_1-\la_k\leq \ell-k$ where $\la_i$ denotes the number of boxes in the $i$th row of $\la$.  The (faithful) Jones-Wenzl representation is the sum: $\rho^{(k,\ell)}=\bigoplus_\la \rho_\la^{(k,\ell)}$.    Wenzl \cite{W} has shown that $\rho^{(k,\ell)}$ is a $C^*$ representation, i.e. the representation space is a Hilbert space (with respect to a Hermitian form induced by the trace $\tr$) and $\rho_\la^{(k,\ell)}(e_i)$ is a self-adjoint operator.
One important consequence is that each $\rho_\la^{(k,\ell}$ induces an irreducible unitary representation of the braid group $\B_n$ via composition with $\sigma_i\rightarrow g_i$, which is also called the Jones-Wenzl representation of $\B_n$.

\section{A Quaternionic Representation}\label{quat}
Consider the $(3,6)$-quotient $\mH_n(3,6)$.  The $(3,6)$-admissible Young diagrams have at most $3$ rows and $\la_1-\la_3\leq 3$.  For $n\geq 3$ there are either $3$ or $4$ Young diagrams of size $n$ that are $(3,6)$-admissible, and  $\eta=\frac{(1-q^{1-3})}{(1+q)(1-q^3)}=1/2$ in this case.
Denote by $\phi_n$ the unitary Jones-Wenzl representation of $\B_n$ induced by $\rho^{(3,6)}$.  Our main goal is to prove the following:
\begin{theorem}\label{mainthm}
The image $\phi_n(\B_n)$ is a finite group.
\end{theorem}
We will prove this theorem by embedding $\mH_n(3,6)$ into a finite dimensional algebra (Lemma \ref{isolemma}) and then showing that the group generated by the images of $g_1,\cdots,g_{n-1}$ is finite (Lemma \ref{finitelemma}).

Denote by $[\;,\;]$ the multiplicative group commutator and let $q=e^{2\pi \im/6}$.
Consider the $\C$-algebra $Q_n$ with generators
$u_1,v_1,\ldots,u_{n-1},v_{n-1}$ subject to the relations:
\begin{enumerate}
 \item[(G1)] $u_i^2=v_i^2=-1$,
\item[(G2)] $[u_i,v_j]=-1$ if $|i-j|\leq 1$,
\item[(G3)] $[u_i,v_j]=1$ if $|i-j|\geq 2$
\item[(G4)] $[u_i,u_j]=[v_i,v_j]=1$
\end{enumerate}
Notice that the group $\{\pm 1,\pm u_i,\pm v_i,\pm u_iv_i\}$ is isomorphic to the group of quaternions.
We see from these relations that $\dim(Q_n)=2^{2n-2}$ since each word in the $u_i,v_i$ has a unique normal form
\begin{equation}\label{normform}
 \pm u_1^{\ep_1}\cdots u_{n-1}^{\ep_{n-1}}v_1^{\nu_1}\cdots v_{n-1}^{\nu_{n-1}}
\end{equation}
with $\nu_i,\ep_i\in\{0,1\}$. Observe that a basis for $Q_n$ is given by taking all $+$ signs in (\ref{normform}).  We define a $\C$-valued trace $\Tr$ on $Q_n$ by setting $\Tr(1)=1$ and $\Tr(w)=0$ for any non-identity word in the $u_i,v_i$. One deduces that $\Tr$ is faithful from the uniqueness of the normal form (\ref{normform}).

Define
\begin{equation}\label{eqgj} s_i:=\frac{-1}{2q}(1+u_i+v_i+u_iv_i)\end{equation}
for $1\leq i\leq n-1$.

\begin{lemma}\label{isolemma}
 The subalgebra $\A_n\subset Q_n$ generated by $s_1,\ldots,s_{n-1}$ is isomorphic to $\mH_n(3,6)$.
\end{lemma}

\begin{proof}

It is a straightforward computation to see that the $s_i$ satisfy
\begin{enumerate}
 \item[(B1)] $s_is_{i+1}s_i=s_{i+1}s_is_{i+1}$
\item[(B2)] $s_js_i=s_is_j$ if $|i-j|\geq 2$
\item[(E1)] $(s_i-q)(s_i+1)=0$
\end{enumerate}
Indeed, relation (B2) is immediate from relations (G3) and (G4).  It is enough to check (B1) and (E1) for $i=1$.  For this we compute:
\begin{eqnarray}
&&s_1^{-1}=-\frac{q}{2}(1-u_1-v_1-u_1v_1)\nonumber\\
&&s_1^{-1}u_1s_1=u_1v_1,\quad s_1^{-1}v_1s_1=u_1,\label{action1}\\
&&s_1^{-1}u_2s_1=u_2v_1,\quad s_1^{-1}v_2s_1=-u_1v_1v_2\nonumber
\end{eqnarray}
from which (B1) and (E1) are deduced.

Thus $\varphi(g_i)=s_i$ induces an algebra homomorphism $\varphi:\mH_n(q)\rightarrow Q_n$ with $\varphi(\mH_n(q))=\A_n$.  Set $f_i:=\varphi(e_i)=\frac{(q-s_i)}{(1+q)}$ and let $b\in Q_{n-1}$ that is, $b$ is in the span of the words in $\{u_1,v_1,\ldots,u_{n-2},v_{n-2}\}$.  The constant term of $f_{n-1}b$ is the product of the constant terms of $b$ and $f_{n-1}$ since $f_{n-1}$ is in the span of $\{1,u_{n-1},v_{n-1},u_{n-1}v_{n-1}\}$, so $\Tr(f_{n-1}b)=\Tr(f_{n-1})\Tr(b)$. For each $a\in\mH_n(q)$ we define $\varphi^{-1}(\Tr)(a):=\Tr(\varphi(a))$, and conclude that $\varphi^{-1}(\Tr)$ is a Markov trace on $\mH_n(q)$.  Computing, we see that $\Tr(f_{n-1})=1/2$ so that by uniqueness $\varphi^{-1}(\Tr)=\tr$ as functionals on $\mH_n(q)$.  Now if $a\in\ker(\varphi)$ we see that $\tr(ac)=\Tr(\varphi(ac))=0$ for any $c$ so that $\ker(\varphi)\subset\Ann(\tr)$.  On the other hand, if $a\in\Ann(\tr)$ we must have $\Tr(\varphi(ac))=\tr(ac)=0$ for all $c\in\mH_n(q)$.  If $\varphi(a)\neq 0$ then, by definition of $\Tr$ and $\varphi$, there exists an $a^\dag\in\mH_n(q)$ such that $\Tr(\varphi(a)\varphi(a^\dag))\neq 0$ since $\Tr$ is faithful.  Therefore  $\Ann(\tr)\subset\ker(\varphi)$.  In particular, we see that $\varphi$ induces:
$$\mH_n(3,6)=\mH_n(q)/\Ann(\tr)\cong \varphi(\mH_n(q))=\A_n\subset Q_n.$$
\end{proof}

\begin{lemma}\label{finitelemma}
The group $G_n$ generated by $s_1,\cdots,s_{n-1}$ is finite.
\end{lemma}
\begin{proof}

  Consider the conjugation action of the $s_i$ on $Q_n$.  We claim that the conjugation action of $s_i$ on the words in the $u_i,v_i$ is by a signed permutation.  Since $s_i$ commutes with words in $u_j,v_j$ with $j\not\in\{i-1,i,i+1\}$, by symmetry it is enough to consider the conjugation action of $s_1$  on the four elements $\{u_1,v_1,u_2,v_2\}$, which is given in (\ref{action1}).

Thus we see that $G_n$ modulo the kernel of this action is a (finite) signed permutation group.
The kernel of this conjugation action lies in the center $Z(Q_n)$ of $Q_n$.  Using the normal form above we find that the center $Z(Q_n)$ is either $1$-dimensional or $4$-dimensional.  Indeed, since the words:
$$W:=\{u_1^{\ep_1}\cdots u_{n-1}^{\ep_{n-1}}v_1^{\nu_1}\cdots v_{n-1}^{\nu_{n-1}}\}$$
for $(\ep_1,\ldots,\ep_{n-1},\nu_1,\ldots,\nu_{n-1})\in\Z_2^{2n-2}$ form a basis for $Q_n$ and $tw=\pm wt$ for $w,t\in W$ we may explicitly compute a basis for the center as those words $w\in W$ that commute with $u_i$ and $v_i$ for all $i$.  This yields two systems of linear equations over $\Z_2$:
\begin{equation}\label{eqnmod2u}
\begin{cases} \ep_1+\ep_2=0,&\\
\ep_{i}+\ep_{i+1}+\ep_{i+2}=0, & 1\leq i\leq n-3 \\
\ep_{n-2}+\ep_{n-1}=0 &\end{cases}
\end{equation} and
\begin{equation}\label{eqnmod2v}
 \begin{cases} \nu_1+\nu_2=0 &\\
\nu_{i-1}+\nu_{i}+\nu_{i+1}=0, & 1\leq i\leq n-3\\
\nu_{n-2}+\nu_{n-1}=0.&\end{cases}
\end{equation}
Non-trivial solutions to (\ref{eqnmod2u}) only exist if $3\mid n$ since we must have $\ep_1=\ep_2=\ep_{n-2}=\ep_{n-1}=1$ as well as $\ep_i=0$ if $3\mid i$ and $\ep_j=1$ if $3\nmid j$ and similarly for (\ref{eqnmod2v}).  Thus $Z(Q_n)$ is $\C$ if $3\nmid n$ and is spanned by $1,U,V$ and $UV$ where
$U=\prod_{3\nmid i} u_i$ and $V=\prod_{3\nmid i} v_i$ if $3\mid n$.
The determinant of the image of $s_i$ under any representation is a $6$th root of unity and hence the same is true for any element $z\in Z(Q_n)\cap G_n$.  Thus for $3\nmid n$ the image of any $z\in Z(Q_n)\cap G_n$ under the left regular representation is a root of unity times the identity matrix, and thus has finite order.  Similarly, if $3\mid n$, the restriction of any $z\in Z(Q_n)\cap G_n$ to any of the four simple components of the left regular representation is a root of unity times the identity matrix and so has finite order.  So the group $G_n$ itself is finite.
\end{proof}
This completes the proof of Theorem \ref{mainthm}.

\begin{remark}  The proof of Lemma \ref{finitelemma} shows that the projective image of $G_n$ is a (non-abelian) subgroup of the full monomial group $G(2,1,4^{n-1})$ of signed $4^{n-1}\times 4^{n-1}$ matrices.  The main goal of this paper is to verify \cite[Conjecture 6.6]{RSW} in this case, but with further effort one could determine the group $G_n$ more precisely.  It is suggested in \cite{LR} that $G_n$ is an extension of $PSU(n-1,\mF_2)$ so that $$|G_n|\approx \frac{1}{3}2^{(n-1)(n-2)/2}\prod_{i=1}^{n-1}(2^i-(-1)^i)$$ but that such a result has not appeared in print.  Modulo the center, the generators $s_i$ have order $3$ so that $G_n/Z(G_n)$ is a quotient of the factor group $\B_n/\lan\sigma_1^3\ra$ (here $\sigma_i$ are the usual generators of $\B_n$).  For $n\leq 5$, Coxeter \cite{cox} has shown that these quotients are finite groups and determined their structure.  In particular, the projective image of $\B_5/\lan \sigma_1^3\ra$ is $PSU(4,\mF_2)$, so $G_5$ is an extension of this simple group.  A strategy for showing $G_n$ is an extension of $PSU(n-1,\mF_2)$ for $n>5$ would be to find an $(n-1)$-dimensional invariant subspace of $Q_n$ so that the restricted action of the braid generators is by order $3$ pseudo-reflections (projectively).  A comparison of the dimensions of the simple $\mH_n(2,6)$-modules with those of $PSU(n-1,\mF_2)$ indicates that one must also restrict to those $n$ not divisible by $3$.
\end{remark}

\section{Concluding Remarks, Questions and Speculations}\label{disc}
The category $\CC(\ssl_3,6)$ does not seem to have any obvious generalizations.  We discuss some of the ways in which $\CC(\ssl_3,6)$ appears to be exceptional by posing a number of (somewhat na\"ive) questions which we expect to have negative answers.
\subsection{Link Invariants}
From any modular category one obtains (quantum) link invariants via Turaev's approach \cite{Tur}.
The link invariant $P_L^\prime(q,\eta)$ associated with $\CC(\ssl_k,\ell)$ is (a variant of) the HOMFLY-PT polynomial (\cite{FYHLMO}, where a different choice of variables is used).  For the choices $q=e^{2\pi \im/6}$ and $\eta=1/2$ corresponding to $\CC(\ssl_3,6)$ the invariant has been identified \cite{LM}:
$$P_L^\prime(e^{2\pi \im/6},1/2)=\pm\im(\sqrt{2})^{\dim H_1(T_L;\Z_2)}$$ where $T_L$ is the triple cyclic cover of the three sphere $S^3$ branched over the link $L$.  There is a similar series of invariants for any odd prime $p$: $\pm\im(\sqrt{p})^{\dim H_1(D_L;\Z_p)}$ where $D_L$ is the double cyclic cover of $S^3$ branched over $L$ (see \cite{LM,GJ}).  It appears that this series of invariants can be obtained from modular categories $\CC(\mathfrak{so}_{p},2p)$.  This has been verified for $p=3,5$ (see \cite{GJ} and \cite{jones89}) and we have recently handled the $p=7$ case (unpublished, using results in \cite{West}).
\begin{question}  Are there modular categories with associated link invariant:
$$\pm\im (\sqrt{p})^{\dim H_1(T_L;\Z_p)}?$$
\end{question}
In \cite{LRW} it is suggested that if the braid group images corresponding to some ribbon category are finite then the corresponding link invariant is \emph{classical}, i.e. equivalent to a homotopy-type invariant.  Another formulation of this idea is found in \cite{twoparas} in which \emph{classical} is interpreted in terms of computational complexity.

\subsection{Fusion Categories and $II_1$ Factors}
The category $\CC(\ssl_3,6)$ is an \emph{integral} fusion category, that is, the simple objects have integral dimensions.  The categories $\CC(\ssl_k,\ell)$ are integral for $(k,\ell)=(3,6)$ and $(k,k+1)$ but no other examples are known (or believed to exist).   $\CC(\ssl_3,6)$ has six simple (isomorphism classes of) objects: $\{X_i,X_i^*\}_{i=1}^3$ of dimension $2$ (dual pairs), three simple objects $\one,Z,Z^*$ of dimension $1$, and one simple object $Y$ of dimension $3$.  The Bratteli diagram for tensor powers of the generating object $X_1$ is given in Figure \ref{brat}.  It is shown in \cite{ENO} that $\CC$ is an integral fusion category if, and only if, $\CC\cong \Rep(H)$ for some semisimple finite dimensional quasi-Hopf algebra $H$, so in particular $\CC(\ssl_3,6)\cong\Rep(H)$ for some quasi-triangular quasi-Hopf algebra $H$.  One wonders if strict coassociativity can be achieved:
\begin{question} Is there a (quasi-triangular) semisimple finite dimensional Hopf algebra $H$ with $\CC(\ssl_3,6)\cong\Rep(H)$?
\end{question}
Other examples of integral categories are the representation categories $\Rep(D^\om G)$ of twisted doubles of finite groups studied in \cite{ERW} (here $G$ is a finite group and $\om$ is a $3$-cocycle on $G$).  Any fusion category $\CC$ with the property that its Drinfeld center $\mathcal{Z}(\CC)$ is equivalent as a braided fusion category to $\Rep(D^\om G)$ for some $\om,G$ is called \emph{group-theoretical} (see \cite{ENO,Nat}).  The main result of \cite{ERW} implies that if $\CC$ is any braided group-theoretical fusion category then the braid group representations obtained from $\CC$ must have finite image.  In \cite{NR} we showed that $\CC(\ssl_3,6)$ is not group-theoretical and in fact has minimal dimension ($36$) among non-group-theoretical integral modular categories.
\begin{question}
Is there a family of non-group-theoretical integral modular categories that includes $\CC(\ssl_3,6)$?
\end{question}
\begin{figure}[t0]
$\xymatrix{ X_1\ar@{-}[d]\ar@{-}[dr]\\  X_1^*\ar@{-}[d]\ar@{-}[dr] & X_2\ar@{-}[d]\ar@{-}[dr] \\
 \one\ar@{-}[d] & Y\ar@{-}[dl]\ar@{-}[d]\ar@{-}[dr] & Z\ar@{-}[d] \\ X_1\ar@{-}[d]\ar@{-}[dr] & X_2^*\ar@{-}[dl]\ar@{-}[dr]& X_3\ar@{-}[dl]\ar@{-}[d]\\ X_1^*\ar@{-}[d]\ar@{-}[dr] & X_2\ar@{-}[d]\ar@{-}[dr] & X_3^*\ar@{-}[dl]\ar@{-}[dr]\\
\one\ar@{-}[d] & Y\ar@{-}[dl]\ar@{-}[dr]\ar@{-}[d] & Z\ar@{-}[d] & Z^*\ar@{-}[dll]\\
X_1&X_2^*&X_3}$
\caption{Bratteli diagram for $\CC(\ssl_3,6)$}\label{brat}
\end{figure}

Notice that $\CC(\ssl_3,6)$ has a ribbon subcategory $\D$ with simple objects $\one, Z,Z^*$ and $Y$.  The fusion rules are the same as those of $\Rep(\mathfrak{A}_4)$: $Y\ot Y\cong\one\oplus Z\oplus Z^*\oplus Y$.  However $\D$ is not symmetric, and $\CC(\ssl_3,6)$ has smallest dimension among modular categories containing $\D$ as a ribbon subcategory (what M\"uger would call a \emph{minimal modular extension} \cite{Mu}).  One possible generalization of  $\CC(\ssl_3,6)$ would be a minimal modular extension of a non-symmetric ribbon category $\D_n$ similar to $\D$ above.  That is, $\D_n$ should be a non-symmetric ribbon category with $n$ $1$-dimensional simple objects $\one=Z_0,\ldots,Z_{n-1}$ and one simple $n$-dimensional object $Y_n$ such that $Y_n\ot Y_n\cong Y_n\oplus\bigoplus_{i=0}^{n-1} Z_i$ and the $Z_i$ have fusion rules like $\Z_n$.  For $\D_n$ to exist even at the generality of fusion categories one must have $n=p^\alpha-1$ for some prime $p$ and integer $\alpha$ by \cite[Corollary 7.4]{EGO}.  However, V. Ostrik \cite{os} informs us that these categories do not admit non-symmetric braidings except for $n=2,3$.  So this does not produce a generalization.

A pair of hyperfinite $II_1$ factors $A\subset B$ with index $[B:A]=4$ can be constructed from $\CC(\ssl_3,6)$ (see \cite[Section 4.5]{wenzlcstar}).  The corresponding principal graph is the Dynkin diagram $E_6^{(1)}$ the nodes of which we label by simple objects:
$$\xymatrix{&& Z^* \ar@{-}[d] \\ && X_3\ar@{-}[d]\\
\one\ar@{-}[r]&X_1\ar@{-}[r]&Y\ar@{-}[r]&X_2^*\ar@{-}[r]& Z}$$
This principal graph can be obtained directly from Bratteli diagram in Figure \ref{brat} as the nodes in the $6$th and $7$th levels and the edges between them.
Hong \cite{Ho} showed that any $II_1$ subfactor pair
$M\subset N$ with principal graph $E_6^{(1)}$ can be constructed from some $II_1$ factor $P$ with an outer action of $\mathfrak{A}_4$ as $M=P\rtimes \Z_3\subset P\rtimes \mathfrak{A}_4=N$.  Subfactor pairs with principal graph $E_7^{(1)}$ and $E_8^{(1)}$ can also be constructed (see eg. \cite{popa}).
We ask:
\begin{question}\label{q:e7e8} Is there a unitary non-group-theoretical integral modular category with principal graph $E_7^{(1)}$ or $E_8^{(1)}$?
\end{question}
Even a braided fusion category with such a principal graph would be interesting, and have interesting braid group image.

Notice that the subcategory $\D$ mentioned above plays a role here as $\mathfrak{A}_4$ corresponds to the Dynkin diagram $E_6^{(1)}$ in the McKay correspondence.  A modular category $\CC$ with principal graph $E_7^{(1)}$ (resp. $E_8^{(1)}$) would contain a ribbon subcategory $\F_1$ (resp. $\F_2$) with the same fusion rules as $\Rep(\mathfrak{S}_4)$ (resp. $\Rep(\mathfrak{A}_5)$).  Using \cite[Lemma 1.2]{EG} we find that such a category $\CC$ must have dimension divisible by $144$ (resp. $3600$).  The ribbon subcategory $\F_2$ must have symmetric braiding (D. Nikshych's proof: $\Rep(\mathfrak{A}_5)$ has no non-trivial fusion subcategories so if it has a non-symmetric braiding, the M\"uger center is trivial.  But if the M\"uger center is trivial it is modular, which fails by \cite[Lemma 1.2]{EG}).  This suggests that for $E_8^{(1)}$ the answer to Question \ref{q:e7e8} is ``no.''  There is a non-symmetric choice for $\F_1$ (as V. Ostrik informs us \cite{os}), with M\"uger center equivalent to $\Rep(\mathfrak{S}_3)$.  By \cite[Prop. 5.1]{Mu} a minimal modular extension $\CC$ of such an $\F_1$ would have dimension $144$.

\end{document}